\documentclass[a4paper,12pt]{amsart}
\usepackage{amssymb}
\usepackage{ifthen}
\usepackage[usenames]{color}
\usepackage{graphicx}
\nonstopmode \numberwithin{equation}{section}
\newtheorem{thm}{Theorem}[section]

\newtheorem{cor}{Corollary}[section]
\newtheorem{lem}{Lemma}[section]
\newtheorem{prop}{Proposition}[section]

\newtheorem{claim}{Claim}[section]
\newtheorem{subclaim}{Subclaim}[section]
\newtheorem{conj}[equation]{Conjecture}
\newtheorem{case}{Case}[section]
\newtheorem*{mysolution}{Solution}
\newtheorem{step}{Step}[section]
\theoremstyle{definition}
\newtheorem{defn}{Definition}[section]
\newtheorem{examp}{Example}[section]
\newtheorem{prob}[equation]{Problem}
\newtheorem{ques}[equation]{Question}
\newtheorem{rem}{Remark}[section]
\newcounter {own}
\def\theown {\thesection       .\arabic{own}}

\newenvironment{pf}[1][]{%
	\vskip 3mm
	\noindent
	\ifthenelse{\equal{#1}{}}%
	{{\slshape Proof. }}%
	{{\slshape #1.} }%
}%
{\qed\bigskip}

\newcounter{alphabet}
\renewcommand{\thealphabet}{\Alph{alphabet}}

\newenvironment{Thm}[1][]{\refstepcounter{alphabet}%
	\bigskip
	\noindent
	{\bf Theorem \thealphabet}%
	\ifthenelse{\equal{#1}{}}{}{ (#1)}%
	{\bf .} \itshape
}{\vskip 8pt}

\newenvironment{Lem}[1][]{\refstepcounter{alphabet}%
	\bigskip
	\noindent
	{\bf Lemma \thealphabet}%
	{\bf .} \itshape
}{\vskip 8pt}

\def\be{\begin{equation}}
	\def\ee{\end{equation}}

\newcommand{\ben}{\begin{enumerate}}
	\newcommand{\een}{\end{enumerate}}

\newcommand{\blem}{\begin{lem}}
	\newcommand{\elem}{\end{lem}}
\newcommand{\bthm}{\begin{thm}}
	\newcommand{\ethm}{\end{thm}}
\newcommand{\bcor}{\begin{cor}}
	\newcommand{\ecor}{\end{cor}}
\newcommand{\beg}{\begin{examp}}
	\newcommand{\eeg}{\end{examp}}
\newcommand{\begs}{\begin{examples}}
	\newcommand{\eegs}{\end{examples}}
\newcommand{\bdefe}{\begin{defn}}
	\newcommand{\edefe}{\end{defn}}
\newcommand{\bprob}{\begin{prob}}
	\newcommand{\eprob}{\end{prob}}
\newcommand{\bques}{\begin{ques}}
	\newcommand{\eques}{\end{ques}}
\newcommand{\bei}{\begin{itemize}}
	\newcommand{\eei}{\end{itemize}}
\newcommand{\bcl}{\begin{claim}}
	\newcommand{\ecl}{\end{claim}}
\newcommand{\bscl}{\begin{subclaim}}
	\newcommand{\escl}{\end{subclaim}}
\newcommand{\bca}{\begin{case}}
	\newcommand{\eca}{\end{case}}
\newcommand{\bstep}{\begin{step}}
	\newcommand{\estep}{\end{step}}
\newcommand{\bsol}{\begin{mysolution}}
	\newcommand{\esol}{\end{mysolution}}
\newcommand{\bcon}{\begin{conj}}
	\newcommand{\econ}{\end{conj}}
\newcommand{\bcons}{\begin{conjs}}
	\newcommand{\econs}{\end{conjs}}
\newcommand{\bprop}{\begin{prop}}
	\newcommand{\eprop}{\end{prop}}
\newcommand{\br}{\begin{rem}}
	\newcommand{\er}{\end{rem}}
\newcommand{\brs}{\begin{rems}}
	\newcommand{\ers}{\end{rems}}
\newcommand{\bo}{\begin{obser}}
	\newcommand{\eo}{\end{obser}}
\newcommand{\bos}{\begin{obsers}}
	\newcommand{\eos}{\end{obsers}}
\newcommand{\bpf}{\begin{pf}}
	\newcommand{\epf}{\end{pf}}
\newcommand{\ba}{\begin{array}}
	\newcommand{\ea}{\end{array}}
\newcommand{\beq}{\begin{eqnarray}}
	\newcommand{\beqq}{\begin{eqnarray*}}
		\newcommand{\eeq}{\end{eqnarray}}
	\newcommand{\eeqq}{\end{eqnarray*}}

\begin{document}
	\title{Riesz--Fej\'er type Inequalities for $\alpha$-Harmonic Functions in the Unit Ball}
	
	\author{Qianyun Li}
	\address{Qianyun Li, Key Laboratory of Computing and Stochastic Mathematics (Ministry of Education), School of Mathematics and Statistics, Hunan Normal University, Changsha, Hunan 410081, P. R. China}
	\email{liqianyun@hunnu.edu.cn}
	
	\author{Shufang Luo}
	\address{Shufang Luo, Key Laboratory of Computing and Stochastic Mathematics (Ministry of Education), School of Mathematics and Statistics, Hunan Normal University, Changsha, Hunan 410081, P. R. China}
	\email{19090204850@163.com}
	
	\author{Zhihao Xu${}^{~\mathbf{*}}$}
	\address{Zhihao Xu, College of Mathematics and Statistics, Hengyang Normal University, Hengyang, Hunan 421010,  P. R. China.}
	\email{734669860@qq.com}
	
	\keywords{Invariant Laplacian equation, Poisson-Szeg\"{o} integral, Riesz--Fej\'er inequality, Schur test.\\
		$^{\mathbf{*}}$Corresponding author}
	
	\subjclass[2020]{Primary 31B05; Secondary 42B30.}
	\begin{abstract}
	In this paper, we establish Riesz--Fej\'er type inequalities for the $\alpha$-harmonic functions $f=P_{\alpha}[f^*]$   in  $\mathbb B^n$, where $f^*\in L^{p}(\mathbb{S}^{n-1})$ and $1<p<\infty$. More precisely, for $n\geq2$ and $\alpha>-1$, we prove the existence of a
	constant $\mathcal{C}_{n,p,\alpha}$ such that
	 $\int_{-1}^{1}
	 |f(r\eta)|^p(1-r^2)^{n-2}\,dr
	 \leq
	 \mathcal{C}_{n,p,\alpha}
	 \int_{\mathbb S^{n-1}}|f^*(\xi)|^p\,d\sigma(\xi)$.
	Moreover, in the range
	$\alpha>\max\left\{
	-\frac{n-1}{p},\,n-2-\frac{2(n-1)}{p}
	\right\}$,
	we determine the sharp constant explicitly. 
	 The result generalize and
	extend the corresponding results of Ahmed et al. (J. Math. Anal. Appl., 563:13, 2026), Hu et al. (Anal. Math. 51:15, 2025) and  Long (arXiv: 2410.12137).
	
	\end{abstract}
	\maketitle
	
	\makeatletter\def\thefootnote{\@arabic\c@footnote}\makeatother

	\section{Introduction  and statement of the main results}
	\subsection{Preliminaries}
	For $n\ge 2$, denote by $\mathbb{R}^{n}$  the usual real Euclidean space of dimension $n$.
	For $x=(x_{1},\ldots,x_{n})\in \mathbb{R}^{n}$, we denote the Euclidean length   of $x$ by $|x|$ and sometimes   identify each point $x$ with a column vector.
	For two column vectors $x,y\in \mathbb{R}^{n}$, the Euclidean inner product of $x$ and $y$ is given by
	$\langle x,y\rangle$.
	For $x_{0}\in \mathbb{R}^{n}$ and $r>0$, let $\mathbb{B}^{n}\left(x_{0}, r\right)=\left\{x \in \mathbb{R}^{n}:\left|x-x_{0}\right|<r\right\}$ and $\mathbb{S}^{n-1}\left(x_{0}, r\right)=\partial \mathbb{B}^{n}\left(x_{0}, r\right)$.
	For convenience, in the following, we denote the unit ball $\mathbb{B}^{n}(0,1)$ by $\mathbb{B}^{n}$, and its boundary  $\mathbb{S}^{n-1}(0,1)$ by $\mathbb{S}^{n-1}$, respectively.
	Moreover, we let $\mathbb{D}$ be unit disk in the complex plane $\mathbb{C}$.
	
	\subsubsection{The Generalized Hypergeometric Series}
	For $a\in \mathbb{R}$ and $k\in\{0,1,2,\ldots\}$, we use $(a )_{k}$ to denote the factorial function
	with $(a )_{0}=1$ ($a\not=0$) and $(a)_{k}=a (a +1) \cdots(a +k-1)$ ($k\geq 1$).
	If $a $ is neither zero nor a negative integer, then
	$$
	(a)_{k}=\frac{\Gamma(a+k) }{ \Gamma(a)} ,
	$$
	where $\Gamma$ is the Gamma function.
	Let $p$, $q=p-1$ be two positive integers and $s\in \mathbb{R}$.
	We define the generalized hypergeometric series by
	\begin{equation}\label{eq-1.1}
		{ }_{p} F_{q}\left(a_{1}, a_{2}, \ldots, a_{p} ; b_{1}, b_{2}, \ldots, b_{q} ; s\right)=\sum_{k=0}^{\infty} \frac{\left(a_{1}\right)_{k} \cdots\left(a_{p}\right)_{k}}{\left(b_{1}\right)_{k} \cdots\left(b_{q}\right)_{k}} \frac{s^{k}}{k !},
	\end{equation}
	where   $a_{i}, b_{j}\in \mathbb{R}$ ($1\leq i\leq p$, $1\leq j\leq q$) and  $b_{j}$  is neither zero nor a negative integer.
	If $\sum_{j=1}^{q}b_{j}-\sum_{i=1}^{p}a_{i}>0$, then the series (\ref{eq-1.1}) converges absolutely in $[-1,1]$ (cf. \cite[ P.~74]{rain}). By \cite[P.~60]{rain}, we know that
	\begin{equation}\label{eq-1.2}
		{ }_{2} F_{1}\left(a, b; c ; s\right)=(1-s)^{c-a-b}{ }_{2} F_{1}\left(c-a, c-b; c ; s\right)
	\end{equation}
	for any $s\in(-1,1)$. 
	If $c-a-b>0$, then it follows from  \cite[P.~ 49]{rain} that
	\begin{equation}\label{eq-1.3}
		{ }_{2} F_{1}\left(a, b; c ; 1\right)=\frac{\Gamma(c)\Gamma(c-a-b)}{\Gamma(c-a)\Gamma(c-b)}.
	\end{equation}
     The Legendre's duplication formula \cite[Theorem 1.5.1]{an1999} states for any $a>0$,
     \begin{equation}\label{eq-1.4}
     	\Gamma(2a)\Gamma(\frac{1}{2})=2^{2a-1}\Gamma(a)\Gamma(a+\frac{1}{2}).
     \end{equation}
	Besides, for any $s,t>0$, the beta function $B(s,t)$ is defined as
	$$
	B(s,t)=\frac{\Gamma(s)\Gamma(t)}{\Gamma(s+t)}.
	$$
	\subsubsection{Invariant  Laplacian  Equation}
	A mapping $f \in C^{2}\left(\mathbb{B}^{n}, \mathbb{R}^{n}\right)(n \geq 2)$ is called to be {\em $\alpha$-harmonic} if it satisfies the  (M\"{o}bius) invariant  Laplacian equation
	$$
	\Delta_{\alpha}f(x)=(1-|x|^{2})\left\{\left(1-|x|^{2}\right) \sum_{j=1}^{n} \frac{\partial^{2}f}{\partial x_{j}^{2}}+2\alpha \sum_{j=1}^{n} x_{j} \frac{\partial f}{\partial x_{j}}+\alpha(n-2-\alpha)f\right\}=0.
	$$
	In this paper, we call $\Delta_{\alpha}$ the (M\"{o}bius) invariant Laplacian, since
	\begin{equation*}
		\Delta_{\alpha}\left\{\left(\operatorname{det} D\psi (x)\right)^{\frac{n-2- \alpha}{2 n}} f(\psi(x))\right\}=\left(\operatorname{det} D\psi (x)\right)^{\frac{n-2-\alpha}{2 n}}\left(\Delta_{\alpha} f\right)(\psi(x))
	\end{equation*}
	for every  $u \in C^{2}\left(\mathbb{B}^n,\mathbb{R}^n\right)$ and for every $\psi \in \mathcal{M }(\mathbb{B}^n )$ (see \cite[Proposition 3.2]{Liu09}).
	Here  $ \mathcal{M} (\mathbb{B}^n )$  denotes the group of those M\"{o}bius transformations that map $\mathbb{B}^{n}$ onto $\mathbb{B}^{n}$.

	When $\alpha=0$ (resp. $\alpha=n-2$), the mapping $u$ is called harmonic (resp. hyperbolic harmonic).
	The properties of these two classes
	of mappings have been investigated extensively by many authors  (cf. \cite{ABR92,Bur92,CH2013, JP99, sto2016}).
	
	The operator $\Delta_{\alpha}$ is closely related to polyharmonic mappings (see \cite{AH2014, Liu21}) and to solutions of the Weinstein equation (see \cite{leu}).
	For a discussion of the fundamental properties of the operator $\Delta_{\alpha}$ in  $\mathbb{D}$, the reader is referred to (\cite{chen23,chen15,k2024,KMM2021, Long2025,Ol14}).
	For the higher-dimensional setting, see (\cite{ACKL2026,chenli2024A,chen24,Liu04, Liu09, ZHD24,zhou22}) and the references therein.

	\subsubsection{Hardy type spaces}
	
	For $p\in(0,\infty]$, the Hardy type space  $\mathcal{H}^p(\mathbb{B}^n, \mathbb{R} )$  consists of all
	measurable mappings $f:\mathbb{B}^n\rightarrow\mathbb{R} $  such that
	$H_{p}(f) < \infty$, where $H_{p}(f)= \sup_{r\in[0,1)} \left\{ M_p(r, f) \right\}$ and
	\begin{align*}
		M_p(r, f) =
		\begin{cases}
			\displaystyle \left( \int_{\mathbb{S}^{n-1}} |f(r\xi)|^p \, d\sigma(\xi) \right)^{\frac{1}{p}}, & \text{if } p \in (0, \infty), \\[6pt]
			\displaystyle \operatorname{esssup}_{\xi \in \mathbb{S}^{n-1}} \left\{ |f(r\xi)| \right\}, & \text{if } p = \infty.
		\end{cases}
	\end{align*}
	Here and hereafter, $d \sigma$ denotes the normalized surface measure
	on   $\mathbb{S}^{n-1}$  so that $\sigma(\mathbb{S}^{n-1}) = 1$ and   $\omega_{n-1}=2\pi^{n/2}/\Gamma(n/2)$ is the $(n-1)$-dimensional surface area of $\mathbb{S}^{n-1}$.
	The norm in   $\mathcal{H}^p(\mathbb{B}^n, \mathbb{R} )$ is denoted by $\| \cdot\|_{\mathcal{H}^p(\mathbb{B}^n, \mathbb{R} )}$,
	which is defined by
	$$
	\| f \|_{\mathcal{H}^p(\mathbb{B}^n, \mathbb{R} )} = H_{p}(f)
	$$
	for every $f\in \mathcal{H}^p(\mathbb{B}^n, \mathbb{R} )$. In particular, we let $H^p(\mathbb{D})$ and $h^p(\mathbb{D})$ denote the Hardy space of holomorphic functions and the harmonic Hardy space on the unit disk $\mathbb{D}$, respectively.

	For $p\in(0,\infty]$, we use $L^{p}\left(\mathbb{S}^{n-1},\mathbb{R} \right)$ to stand for the space of all measurable mappings  $f:\mathbb{S}^{n-1}\rightarrow\mathbb{R} $  such that $M_{p}(1,f)<\infty$. The norm in $L^{p}\left(\mathbb{S}^{n-1},\mathbb{R} \right)$ is denoted by  $\|\cdot\|_{L^{p}\left(\mathbb{S}^{n-1},\mathbb{R} \right)}$,
	which is defined by
	$$
	\|f\|_{L^{p}\left(\mathbb{S}^{n-1},\mathbb{R} \right)}=M_{p}(1,f)
	$$
	for $f\in L^{p}\left(\mathbb{S}^{n-1},\mathbb{R} \right)$.

	If $\phi\in L^{p}(\mathbb{S}^{n-1},\mathbb{R})$ with $p\in[1,\infty]$,  we define the invariant Poisson integral or   Poisson-Szeg\"{o} integral
	of $\phi$ in $\mathbb{B}^{n}$ by $P_{\alpha}[\phi]$, where
	\begin{equation*}
		P_{\alpha}[\phi](x)=\int_{\mathbb{S}^{n-1}} P_{\alpha}(x, \zeta) \phi(\zeta) d \sigma(\zeta),
	\end{equation*}
	
	\begin{equation*}
		P_{\alpha}(x, \zeta)=C_{n, \alpha} \frac{\left(1-|x|^{2}\right)^{1+ \alpha}}{|x-\zeta|^{n+ \alpha}}
		\;\;\text{and}\;\;C_{n, \alpha}=\frac{\Gamma\left(\frac{n+\alpha}{2}\right) \Gamma(1+\frac{\alpha}{2})}{\Gamma\left(\frac{n}{2}\right) \Gamma(1+ \alpha)}
	\end{equation*}
	(cf. \cite[Section 1]{li24}). Specifically, we know that $P_{h}(x, \zeta):=P_{n-2}(x, \zeta)$ is Poisson-Szeg\"{o} integral about hyperbolic harmonic.
	
	In \cite{Liu04}, Liu and Peng  investigated  the solvability of the
	Dirichlet problem associated with the (M\"{o}bius) invariant Laplacian
	\begin{equation}\label{eq-1.5}
		\left\{\begin{array}{ll}
			\Delta_{\alpha} f(x)=0, & \text{ if} \;x\in\mathbb{B}^{n}, \\
			f(\zeta)=\phi(\zeta),& \text{ if}\; \zeta\in\mathbb{S}^{n-1},
		\end{array}\right.
	\end{equation}
	where $\phi\in C ( \mathbb{S}^{n-1} ,\mathbb{ R})$.
	They showed that the Dirichlet problem \eqref{eq-1.5} has a solution if and only if $\alpha>-1$ (see \cite[Theorem 2.4]{Liu04}).
	In this case, the solution is unique and
	$$
	f(x)= P_{\alpha}[\phi](x).	
	$$
	
	\subsection{Main results}
	
	The classical Riesz--Fej\'er inequality \cite{FR1921} states that if $ f\in H^p(\mathbb{D})$ for $0 < p < \infty$, then
	\[
	\int_{-1}^{1} |f(x)|^p dx \leq \frac12 \int_{0}^{2\pi} |f(e^{i\theta})|^p d\theta,
	\]
	with the constant $\frac12$ is sharp constant.
	
	In 2020, Kayumov, Ponnusamy and Kaliraj \cite{KMM2021} initiated the study of the analogous inequality for the harmonic Hardy space $h^p(\mathbb{D})$. They proved that for $1 < p < \infty$ and $f \in h^p(\mathbb{D})$,
	\[
	\int_{-1}^{1} |f(re^{is})|^p dr \leq C_p \int_{0}^{2\pi} |f(e^{i\theta})|^p d\theta,
	\]
	with $C_p = 1$ for $p \geq 2$ and $C_p = 1/\bigl(2\cos^p(\pi/(2p))\bigr)$ for $1 < p \leq 2$. 
	Melentijevi\'c and Bo\v zin \cite{MB2021} later showed that the same constant $1/\bigl(2\cos^p(\pi/(2p))\bigr)$ is sharp for the entire range $1 < p < \infty$.
	
	Recently,  Ahmed and Khalfallah \cite{AK2026} established a Riesz--Fej\'er type inequality for hyperbolic harmonic functions in the unit ball $\mathbb{B}^{n}$. They showed
	the following:
	
	\begin{Thm}{\rm (\cite[Theorem 1.1]{AK2026})}
		Let $n \geq 2$ and $1 < p < \infty$. Let $f=P_{h}[f^*]$ with $f^*\in L^{p}(\mathbb{S}^{n-1},\mathbb{ R})$. Then, for every $\eta \in \mathbb{S}^{n-1}$,
		\begin{equation}\label{eq-1.6}
			\int_{-1}^{1} |f(r\eta)|^p (1-r^2)^{n-2} dr \leq \mathcal{C}_{n,p} \int_{\mathbb{S}^{n-1}} |f^*(\xi)|^p d\sigma(\xi),
		\end{equation}
		where
		$$
		\mathcal{C}_{n,p}=2^{(p+1)(n-2)} \left(\frac{\omega_{n-2}}{\omega_{n-1}}\right)^{p-1}B\left( \frac{(n-1)(p-1)}{2p},\frac{(n-1)(p+1)}{2p} \right)^p.
		$$
		The constant $\mathcal{C}_{n,p}$ is sharp. 
	\end{Thm}
	When $n = 2$, obviously, hyperbolic harmonic mappings coincide with harmonic mappings and $\mathcal{C}_{2,p}=\frac{\pi}{\cos^p(\pi/(2p))}$. Since $d\sigma(e^{i\theta})=d\theta/(2\pi)$,   \eqref{eq-1.6} corresponds to the sharp result due to Melentijevi\'c and Bo\v zin.
	
	For the  Riesz--Fej\'er inequality of $\alpha$-harmonic functions in $\mathbb{D}$, Long \cite{Long2024} obtained an existence result for $-1<\alpha\leq0$ and established an asymptotically sharp result as $\alpha\to0$, while Hu et al. \cite{HFS2025} obtained an existence result for $\alpha\geq0$ and also established an asymptotically sharp result as $\alpha\to0$.

	The main aim of this paper is to generalize  \cite{HFS2025,Long2024}  to the higher-dimensional case. Our results are as follows.
	
	\begin{thm}\label{thm-1.1}
		Let $n\geq 2$, $1<p<\infty$ and $\alpha>-1$. If $f=P_{\alpha}[f^*]$  in $\mathbb{B}^{n}$  with $f^*\in L^p(\mathbb S^{n-1},\mathbb{R})$,
		then there exists a  positive constant  $\mathcal{C}_{n,p,\alpha}$ such that for every $\eta\in\mathbb S^{n-1}$,
		$$
		\int_{-1}^{1}
		|f(r\eta)|^p(1-r^2)^{n-2}\,dr
		\leq\mathcal{C}_{n,p,\alpha}
		\int_{\mathbb S^{n-1}}|f^*(\xi)|^p\,d\sigma(\xi).
		$$
		Moreover, if
		$$
		\alpha>\max\left\{
		-\frac{n-1}{p},
		\,n-2-\frac{2(n-1)}{p}
		\right\},
		$$
		then the sharp constant is
		$$
		\mathcal{C}_{n,p,\alpha}
		=2^{p\alpha+n-2}	C_{n,\alpha}^{p}
		\left(\frac{\omega_{n-2}}{\omega_{n-1}}\right)^{p-1}\left[
		B\left(\frac{(n-1)(p-1)}{2p},
		\frac{\alpha+1+\frac{n-1}{p}}{2}\right)
		\right]^p.
		$$
	\end{thm}
	
	\begin{rem}
		Taking $\alpha=n-2$, it is easy to see that $\alpha>
		\max\left\{-\frac{n-1}{p},
		\,n-2-\frac{2(n-1)}{p}\right\}$. Substitute $\alpha = n-2$ into $\mathcal{C}_{n,p,\alpha}$, then yields $\mathcal{C}_{n,p}$. When 
		$n=2$, our results improve upon those of Long \cite[Theorem 1.6 ]{Long2024} and Hu et al. \cite[Theorem 1.1]{HFS2025}.
	\end{rem}
	
	The rest of this paper is organized as follows.
	In Sections  \ref{Sec-2}, we recall some necessary terminology and known results and prove the required lemma.
	In  Section \ref{Sec-3},  we give the proof of Theorem \ref{thm-1.1}.

	\section{Preliminaries}\label{Sec-2}
In this section, we introduce some
	necessary terminology. Then, we prove some auxiliary results.

	\begin{Lem}{\rm (\cite[P. 55]{MOS1966})}\label{Lem-B}
		Let $\mu>0$, $\nu\in\mathbb{R}$, and $|r|<1$. Then
		\begin{equation*}
			\int_{0}^{\pi}
			\frac{\sin^{2\mu-1}\theta}
			{\bigl(1-2r\cos\theta+r^{2}\bigr)^{\nu}}
			\,d\theta
			=
			B\left(\mu,\frac12\right)
			{}_2F_1\left(
			\nu,\nu-\mu+\frac12;\mu+\frac12;r^{2}
			\right).
		\end{equation*}
	\end{Lem}
	
	\begin{Lem}{\rm (\cite[Lemma 1.2]{Ol14})}\label{Lem-B-2}
			Let  $c>0$, $a \leq c$, $b \leq c$  and  $a b \leq 0$ ($a b \geq 0$) . Then the function  ${}_{2}F_{1}(a, b ; c ; \cdot)$  is decreasing (increasing) on $( 0,1 )$.
	\end{Lem}

	\begin{Lem}{\rm (\cite[Equation 2.5]{AK2026} or \cite[Theorem A.5]{ABR92})}\label{Lem-C}
		Let $f$ be a measurable function of one real variable. If $\eta\in \mathbb{S}^{n-1}$ and
		$\theta=\arccos\langle\zeta,\eta\rangle\in[0,\pi]
		$,
		then
		\begin{equation*}
			\int_{\mathbb{S}^{n-1}}
			f\bigl(\langle\zeta,\eta\rangle\bigr)
			\,d\sigma(\zeta)
			=
			\frac{\omega_{n-2}}{\omega_{n-1}}
			\int_{0}^{\pi}
			(\sin\theta)^{n-2}f(\cos\theta)
			\,d\theta.
		\end{equation*}
	\end{Lem}
	
	\begin{Lem}{\rm (\cite[Theorem 3]{HS1990})}\label{Lem-D}
		Suppose $(X,\mu)$ and $(Y,\nu)$ be $\sigma$-finite measure spaces. Let $K\geq0$ be measurable, and define
		$$
		Tf(x)=\int_Y K(x,y)f(y)\,d\nu(y),
		\qquad
		T^*g(y)=\int_X K(x,y)g(x)\,d\mu(x).
		$$
		Let $1<p<\infty$. Suppose that there exist a measurable function $h>0$ finite a.e. on $Y$, and a constant $\Lambda>0$ such that
		\begin{equation*}
			T^*\big((Th)^{p-1}\big)(y)
			\leq
			\Lambda h(y)^{p-1}
			\quad\text{for }\nu\text{-a.e. }y\in Y.
		\end{equation*}
		Then
		\begin{equation*}
			\int_X |Tf(x)|^p\,d\mu(x)
			\leq
			\Lambda
			\int_Y |f(y)|^p\,d\nu(y),
			\qquad f\in L^p(Y,\nu).
		\end{equation*}
		Consequently,
		$$
		\|T\|_{L^p(Y,\nu)\to L^p(X,\mu)}\leq \Lambda^{1/p}.
		$$
	\end{Lem}
	
\begin{lem}\label{lem-2.1}
	Let $n\geq2$, $1<p<\infty$, and $\alpha>-1$. For $0<\theta<\pi$, define
	$$
	\mathcal I_{n,p,\alpha}(\theta)
	=
	\int_{-1}^{1}
	\frac{(1-r^2)^{\alpha+\frac{n-1}{p}}}
	{(1-2r\cos\theta+r^2)^{(n+\alpha)/2}}
	\,dr.
	$$
	Then there exists a positive constant $C_{1}:=C_{1}(n,p,\alpha)$ such that
	$$
	\mathcal I_{n,p,\alpha}(\theta)
	\leq
	C_{1}
	(\sin\theta)^{-\frac{(n-1)(p-1)}{p}},
	\qquad 0<\theta<\pi,
	$$
	where
	$$
	C_{1}=\begin{cases}
		\displaystyle
		\frac{\sqrt{\pi}\,\Gamma\left(\frac{n-1+(1+\alpha)p}{2p}\right)}
		{2\Gamma\left(\frac{n-1+(2+\alpha)p}{2p}\right)},
		& \text{if } \alpha-n+2+\frac{2(n-1)}{p}<0, \\[2ex]
		\displaystyle
		2^{\alpha+\frac{n-1}{p}-1}
		B\left(
		\frac{(n-1)(p-1)}{2p},
		\frac{\alpha+1+\frac{n-1}{p}}{2}
		\right),
		& \text{if } \alpha-n+2+\frac{2(n-1)}{p}\geq0.
	\end{cases}
	$$
\end{lem}

\begin{proof}
	Fix $0<\theta<\pi$. We make the change of variables
	$$
	\frac{1+r}{1-r}
	=
	y\cot\frac{\theta}{2},
	\qquad y\in(0,\infty).
	$$
	Equivalently,
	$$
	r=\frac{y\cos\frac{\theta}{2}-\sin\frac{\theta}{2}}
	{y\cos\frac{\theta}{2}+\sin\frac{\theta}{2}}.
	$$
	A direct computation gives
	$$
	1-r^2=\frac{2y\sin\theta}{\left(y\cos\frac{\theta}{2}
	+\sin\frac{\theta}{2}\right)^2},
	$$
	$$
	1-2r\cos\theta+r^2=\frac{(1+y^2)\sin^2\theta}
	{\left(y\cos\frac{\theta}{2}+\sin\frac{\theta}{2}\right)^2}
	$$
	and
	$$
	dr=\frac{\sin\theta}
	{\left(y\cos\frac{\theta}{2}+\sin\frac{\theta}{2}\right)^2
	}\,dy.
	$$
	Substituting these identities into the definition of
	$\mathcal I_{n,p,\alpha}(\theta)$, we obtain
	$$
	\mathcal I_{n,p,\alpha}(\theta)=
	2^{\alpha+\frac{n-1}{p}}
	(\sin\theta)^{-\frac{(n-1)(p-1)}{p}}
	\mathcal G_{n,p,\alpha}(\theta),
	$$
	where
	$$
	\mathcal G_{n,p,\alpha}(\theta)
	=\int_0^\infty\frac{y^{\alpha+\frac{n-1}{p}}}
	{(1+y^2)^{(n+\alpha)/2}}
	\frac{dy}{\left(y\cos\frac{\theta}{2}+\sin\frac{\theta}{2}
	\right)^{\alpha-n+2+\frac{2(n-1)}{p}}}.
	$$
	Now put $y=\tan x$, where $0<x<\pi/2$. Since
	$$
	y\cos\frac{\theta}{2}+\sin\frac{\theta}{2}
	=\frac{\sin\left(x+\frac{\theta}{2}\right)}{\cos x},
	$$
	we have
	\begin{equation}\label{eq-2.1}
		\mathcal G_{n,p,\alpha}(\theta)
		=\int_0^{\pi/2}\frac{\sin^{\alpha+\frac{n-1}{p}}x\cos^{\alpha+\frac{n-1}{p}}x}
		{\sin^{\alpha-n+2+\frac{2(n-1)}{p}}
	\left(x+\frac{\theta}{2}\right)
		}\,dx.
	\end{equation}
	
		In order to estimate $\mathcal I_{n,p,\alpha}(\theta)$, we divide the proof into two cases according to the
	value of $\alpha$.
	
	\noindent\textbf{Case 1.}
	Suppose that $\alpha-n+2+\frac{2(n-1)}{p}<0$.
	
	Based on the interval of $\alpha$, one obtains that
	$$
	\frac{1}
	{\sin^{\alpha-n+2+\frac{2(n-1)}{p}}\left(x+\frac{\theta}{2}\right)}
	=\sin^{-\alpha+n-2-\frac{2(n-1)}{p}}
	\left(x+\frac{\theta}{2}\right)
	\leq 1.
	$$
	Since $\alpha+\frac{n-1}{p}>-1$,
	it follows from \eqref{eq-2.1} that
	\begin{align*}
		\mathcal G_{n,p,\alpha}(\theta)
		&\leq\int_0^{\pi/2}\sin^{\alpha+\frac{n-1}{p}}x
		\cos^{\alpha+\frac{n-1}{p}}x
		\,dx  \\
		&=\frac{2^{-\frac{n-1+(1+\alpha)p}{p}}
	\sqrt{\pi}\,\Gamma\left(\frac{n-1+(1+\alpha)p}{2p}\right)}{\Gamma\left(\frac{n-1+(2+\alpha)p}{2p}\right)
		}.
	\end{align*}
	Consequently,
	$$
	\mathcal I_{n,p,\alpha}(\theta)\leq
	\frac{\sqrt{\pi}\,\Gamma\left(\frac{n-1+(1+\alpha)p}{2p}\right)}{2\Gamma\left(\frac{n-1+(2+\alpha)p}{2p}\right)}(\sin\theta)^{-\frac{(n-1)(p-1)}{p}}.
	$$
	
	\noindent\textbf{Case 2.}
	Suppose that $\alpha-n+2+\frac{2(n-1)}{p}\geq0$.
	
	Put
	$$
	\beta=\alpha-n+2+\frac{2(n-1)}{p}.
	$$
	Then \eqref{eq-2.1} becomes
	$$
	\mathcal G_{n,p,\alpha}(\theta)
	=\int_0^{\pi/2}\sin^{\alpha+\frac{n-1}{p}}x
	\cos^{\alpha+\frac{n-1}{p}}x\sin^{-\beta}
	\left(x+\frac{\theta}{2}\right)\,dx.
	$$
	We first justify differentiating under the integral sign. Let
	$J\subset(0,\pi)$ be a compact interval. Then there exists a constant
	$c_J>0$ such that
	$$
	\sin\left(x+\frac{\theta}{2}\right)\geq c_J,
	\qquad
	0<x<\frac{\pi}{2},\quad \theta\in J.
	$$
	Hence, for $j=0,1,2$, there exists a constant $C_J>0$ such that
	$$
	\left|\frac{\partial^j}{\partial\theta^j}
	\sin^{-\beta}\left(x+\frac{\theta}{2}\right)
	\right|\leq C_J,
	\qquad0<x<\frac{\pi}{2},\quad \theta\in J.
	$$
	Since $\alpha+\frac{n-1}{p}>-1$,
	the function $\sin^{\alpha+\frac{n-1}{p}}x
	\cos^{\alpha+\frac{n-1}{p}}x$
	is integrable on $(0,\pi/2)$. Therefore, by the dominated convergence
	theorem, differentiation under the integral sign is justified on
	$(0,\pi)$.
	
	For $0<\theta<\pi$, we get
	\begin{align*}
		\mathcal G_{n,p,\alpha}''(\theta)
		&=\frac{\beta}{4}\int_0^{\pi/2}
		\sin^{\alpha+\frac{n-1}{p}}x
		\cos^{\alpha+\frac{n-1}{p}}x  \\
		&\quad\times
		\frac{1+\beta\cos^2\left(x+\frac{\theta}{2}\right)}{\sin^{\beta+2}\left(x+\frac{\theta}{2}\right)}\,dx\geq0.
	\end{align*}
	Thus $\mathcal G_{n,p,\alpha}$ is convex on $(0,\pi)$.
	
	We next show that $\mathcal G_{n,p,\alpha}$ admits a continuous
	extension to $[0,\pi]$ and compute the two endpoint values.
	As $\theta\to0^+$, the integrand in \eqref{eq-2.1}
	converges pointwise to
	$$
	\sin^{
		n-2-\frac{n-1}{p}
	}x
	\cos^{
		\alpha+\frac{n-1}{p}
	}x.
	$$
	Moreover, this function is integrable on $(0,\pi/2)$ because $n-2-\frac{n-1}{p}>-1$
	and $\alpha+\frac{n-1}{p}>-1$.
	Therefore
	\begin{align*}
		\mathcal G_{n,p,\alpha}(0)
		&:=\lim_{\theta\to0^+}
		\mathcal G_{n,p,\alpha}(\theta) \\
		&=\int_0^{\pi/2}\sin^{n-2-\frac{n-1}{p}
		}x\cos^{\alpha+\frac{n-1}{p}}x\,dx  \\
		&=\frac12B\left(\frac{(n-1)(p-1)}{2p},
		\frac{\alpha+1+\frac{n-1}{p}}{2}
		\right).
	\end{align*}
	Similarly,
	\begin{align*}
		\mathcal G_{n,p,\alpha}(\pi)
		&:=\lim_{\theta\to\pi^-}
		\mathcal G_{n,p,\alpha}(\theta) \\
		&=\int_0^{\pi/2}\sin^{\alpha+\frac{n-1}{p}}x
		\cos^{n-2-\frac{n-1}{p}}x\,dx  \\
		&=\frac12B\left(\frac{\alpha+1+\frac{n-1}{p}}{2},
		\frac{(n-1)(p-1)}{2p}\right).
	\end{align*}
	By the symmetry of the Beta function,
	$$
	\mathcal G_{n,p,\alpha}(0)
	=
	\mathcal G_{n,p,\alpha}(\pi).
	$$
	Since $\mathcal G_{n,p,\alpha}$ is convex on $(0,\pi)$ and admits a
	continuous extension to $[0,\pi]$, it follows that
	$$
	\mathcal G_{n,p,\alpha}(\theta)
	\leq\max\{\mathcal G_{n,p,\alpha}(0),
	\mathcal G_{n,p,\alpha}(\pi)\}
	=\mathcal G_{n,p,\alpha}(0),
	\qquad 0<\theta<\pi.
	$$
	Therefore
	$$
    \mathcal G_{n,p,\alpha}(\theta)\leq
	\frac12B\left(\frac{(n-1)(p-1)}{2p},
	\frac{\alpha+1+\frac{n-1}{p}}{2}\right),
	\qquad 0<\theta<\pi.
	$$
	Consequently,
	$$
	\mathcal I_{n,p,\alpha}(\theta)
	\leq2^{\alpha+\frac{n-1}{p}-1}
	B\left(\frac{(n-1)(p-1)}{2p},
	\frac{\alpha+1+\frac{n-1}{p}}{2}\right)
	(\sin\theta)^{-\frac{(n-1)(p-1)}{p}}.
	$$
	The proof is complete.
\end{proof}

	\section{Proof of Theorem \ref{thm-1.1}}\label{Sec-3}
	The aim of this section is to prove Theorem \ref{thm-1.1}. The proof of Theorem \ref{thm-1.1} will be split into two steps for discussion.
	
	\begin{step}
		Boundedness.
	\end{step}
	
We apply the Schur test in the following setting. Let $X=(-1,1)$ be equipped with the measure
	$$
	d\mu_n(r)=(1-r^2)^{n-2}\,dr,
	$$
	and let $Y=\mathbb S^{n-1}$ be equipped with the normalized surface measure $d\sigma(\xi)$. For fixed $\eta\in\mathbb S^{n-1}$, define the positive integral operator
	$$
	(T_{\alpha,\eta}f)(r)
	=C_{n, \alpha}\int_{\mathbb S^{n-1}}
	\frac{(1-r^2)^{1+\alpha}}
	{(1-2r\langle\eta,\xi\rangle+r^2)^{(n+\alpha)/2}}
	f(\xi)\,d\sigma(\xi),
	\qquad -1<r<1.
	$$
	Its adjoint is given by
	$$
	(T_{\alpha,\eta}^*g)(\xi)
	=\int_{-1}^{1}P_\alpha(r\eta,\xi)g(r)(1-r^2)^{n-2}\,dr.
	$$
	
	Define
	$$
	h(\xi)=\left(1-\langle\xi,\eta\rangle^2\right)^{-\frac{n-1}{2p}}.
	$$
	If $\theta=\arccos\langle\xi,\eta\rangle$,
	then
	$$
	h(\xi)=(\sin\theta)^{-\frac{n-1}{p}}.
	$$
	Since
	$$
	0<\frac{n-1}{p}<n-1,
	$$
	the function $h$ is finite almost everywhere and integrable on $\mathbb S^{n-1}$.
	
	By the Lemma \ref{Lem-C}, we have
	$$
	\begin{aligned}
		(T_{\alpha,\eta}h)(r)
		&=C_{n,\alpha}\frac{\omega_{n-2}}{\omega_{n-1}}(1-r^2)^{1+\alpha}\int_0^\pi\frac{(\sin\theta)^{n-2-\frac{n-1}{p}}}{(1-2r\cos\theta+r^2)^{(n+\alpha)/2}}
		\,d\theta.
	\end{aligned}
	$$
	Using Lemma \ref{Lem-B} with
	$$
	\mu=\frac{(n-1)(p-1)}{2p},
	\qquad\nu=\frac{n+\alpha}{2},
	$$
	we obtain
	$$
	\begin{aligned}
		(T_{\alpha,\eta}h)(r)
		&=C_{n, \alpha}\frac{\omega_{n-2}}{\omega_{n-1}}
		B\left(\frac{(n-1)(p-1)}{2p},\frac12\right)
		(1-r^2)^{1+\alpha}  \\
		&\quad\times
		{}_2F_1\left(\frac{n+\alpha}{2},\frac{\alpha+2+\frac{n-1}{p}}{2};\frac{n-\frac{n-1}{p}}{2};r^2\right).
	\end{aligned}
	$$
	Form \eqref{eq-1.2}, we get
	\begin{align*}
		&(1-r^2)^{\frac{n-1}{p}}(T_{\alpha,\eta}h)(r) \\
		&=C_{n, \alpha}\frac{\omega_{n-2}}{\omega_{n-1}}
		B\left(\frac{(n-1)(p-1)}{2p},\frac12\right)
		{}_2F_1\left(-\frac{\alpha+\frac{n-1}{p}}{2},
		\frac{n-2-\alpha-\frac{2(n-1)}{p}}{2};
		\frac{n-\frac{n-1}{p}}{2};r^2\right).
	\end{align*}
	Moreover,
	$$
	\frac{n-\frac{n-1}{p}}{2}+\frac{\alpha+\frac{n-1}{p}}{2}
	-\frac{n-2-\alpha-\frac{2(n-1)}{p}}{2}
	=\alpha+1+\frac{n-1}{p}>0.
	$$
	Thus ${}_2F_1\left(
	-\frac{\alpha+\frac{n-1}{p}}{2},
	\frac{n-2-\alpha-\frac{2(n-1)}{p}}{2};
	\frac{n-\frac{n-1}{p}}{2};t\right)$ is bound on $[0,1]$. Let 
	$$
	C_{2}:=\sup\left\{{}_2F_1\left(
	-\frac{\alpha+\frac{n-1}{p}}{2},\ 
	\frac{n-2-\alpha-\frac{2(n-1)}{p}}{2}\ ;\ 
	\frac{n-\frac{n-1}{p}}{2}\ ;\ t
	\right):\ t\in[0,1]\right\}
	$$
	Then
	$$
	(T_{\alpha,\eta}h)(r)\leq
	C_{n, \alpha}C_{2}\frac{\omega_{n-2}}{\omega_{n-1}}
	B\left(\frac{(n-1)(p-1)}{2p},
	\frac12\right)(1-r^2)^{-\frac{n-1}{p}},
	\qquad 0<r<1.
	$$
	Since
	$$
	P_\alpha(-r\eta,\xi)=P_\alpha(r\eta,-\xi)\qquad
	\text{and}\qquad
	h(-\xi)=h(\xi),
	$$
	the same estimate holds for $-1<r<0$. Consequently,
	\begin{equation}\label{eq-3.1}
		(T_{\alpha,\eta}h)(r)\leq C_{n, \alpha}C_{2}\frac{\omega_{n-2}}{\omega_{n-1}}
		B\left(\frac{(n-1)(p-1)}{2p},
		\frac12\right)(1-r^2)^{-\frac{n-1}{p}},
		\qquad -1<r<1.
	\end{equation}
	
	From \eqref{eq-3.1}, we have
	\begin{equation}\label{eq-3.2}
		T_{\alpha,\eta}^*
		\left((T_{\alpha,\eta}h)^{p-1}\right)(\xi)
		\leq C_{n,\alpha}\left(C_{n, \alpha}C_{2}\frac{\omega_{n-2}}{\omega_{n-1}}B\left(\frac{(n-1)(p-1)}{2p},
		\frac12\right)\right)^{p-1}J(\theta),
	\end{equation}
	where
	$$
	J(\theta)
	=
	\int_{-1}^{1}
	\frac{(1-r^2)^{\alpha+\frac{n-1}{p}}}
	{(1-2r\cos\theta+r^2)^{(n+\alpha)/2}}
	\,dr.
	$$
	
	By Lemma \ref{lem-2.1}, we know that
	\begin{equation}\label{eq-3.3}
		J(\theta)\leq C_{1}(\sin\theta)^{-\frac{(n-1)(p-1)}{p}},
		\qquad 0<\theta<\pi.
	\end{equation}
	Combining \eqref{eq-3.2} and \eqref{eq-3.3}, we obtain
	$$
	T_{\alpha,\eta}^*
	\left((T_{\alpha,\eta}h)^{p-1}\right)(\xi)
	\leq
	C_{1}C_{n,\alpha}^{p}\left(C_{2}\frac{\omega_{n-2}}{\omega_{n-1}}
	B\left(
	\frac{(n-1)(p-1)}{2p},
	\frac12
	\right)\right)^{p-1}h(\xi)^{p-1}
	$$
	for $\sigma$-almost every $\xi\in\mathbb S^{n-1}$. Therefore,  Lemma \ref{Lem-D} gives
	$$
	\int_{-1}^{1}
	|T_{\alpha,\eta}f(r)|^p(1-r^2)^{n-2}\,dr
	\leq
	C_{1}C_{n,\alpha}^{p}\left(C_{2}\frac{\omega_{n-2}}{\omega_{n-1}}
	B\left(
	\frac{(n-1)(p-1)}{2p},
	\frac12
	\right)\right)^{p-1}
	\int_{\mathbb S^{n-1}}|f(\xi)|^p\,d\sigma(\xi).
	$$
	This proves the asserted boundedness.
	
	\begin{step}
		Sharpness of the constant.
	\end{step}
	
	Assume now that
	$$
	\alpha>\max\left\{-\frac{n-1}{p},
	\,n-2-\frac{2(n-1)}{p}
	\right\}.
	$$
	Then
	$$
	-\frac{\alpha+\frac{n-1}{p}}{2}<0
	\qquad
	\text{and}\qquad
	\frac{n-2-\alpha-\frac{2(n-1)}{p}}{2}<0.
	$$
	Hence, by the Lemma \ref{Lem-B-2}, ${}_2F_1
	\left(
	-\frac{\alpha+\frac{n-1}{p}}{2},
	\frac{n-2-\alpha-\frac{2(n-1)}{p}}{2};
	\frac{n-\frac{n-1}{p}}{2};
	t
	\right)$ is increasing in $t\in[0,1)$, then
	$$
	\sup_{0<r<1}
	(1-r^2)^{\frac{n-1}{p}}(T_{\alpha,\eta}h)(r) 
	=\lim_{r\to1^-}(1-r^2)^{\frac{n-1}{p}}(T_{\alpha,\eta}h)(r).
	$$
	By \eqref{eq-1.3} and \eqref{eq-1.4}, we obtain
	\begin{align*}
	&\lim_{r\to1^-}
	(1-r^2)^{\frac{n-1}{p}}(T_{\alpha,\eta}h)(r) \\
	=&C_{n, \alpha}\frac{\omega_{n-2}}{\omega_{n-1}}\frac{\Gamma\left(\frac{1}{2}\right)\Gamma\left(\frac{(n-1)(p-1)}{2p}\right)}{\Gamma\left(\frac{n(p-1)+1}{2p}\right)}\frac{\Gamma\left(\frac{n(p-1)+1}{2p}\right)\Gamma\left(1+\alpha+\frac{n-1}{p}\right)}{\Gamma\left(\frac{n+\alpha}{2}\right)\Gamma\left(\frac{2+\alpha+\frac{n-1}{p}}{2}\right)}\\
	=&C_{n,\alpha}\frac{\omega_{n-2}}{\omega_{n-1}}
	2^{\alpha+\frac{n-1}{p}}
	B\left(
	\frac{(n-1)(p-1)}{2p},
	\frac{\alpha+1+\frac{n-1}{p}}{2}
	\right).
\end{align*}
	Thus
	\begin{equation}\label{eq-3.4}
		(T_{\alpha,\eta}h)(r)
		\leq C_{n,\alpha}\frac{\omega_{n-2}}{\omega_{n-1}}2^{\alpha+\frac{n-1}{p}}
		B\left(\frac{(n-1)(p-1)}{2p},
		\frac{\alpha+1+\frac{n-1}{p}}{2}\right)
		(1-r^2)^{-\frac{n-1}{p}}.
	\end{equation}
	Furthermore, the assumption $\alpha>n-2-\frac{2(n-1)}{p}$
	implies
	$$
	\alpha-n+2+\frac{2(n-1)}{p}>0.
	$$
	Therefore the  part of Lemma \ref{lem-2.1} gives
	$$
	J(\theta)\leq
	2^{\alpha+\frac{n-1}{p}-1}B\left(\frac{(n-1)(p-1)}{2p},\frac{\alpha+1+\frac{n-1}{p}}{2}\right)
	(\sin\theta)^{-\frac{(n-1)(p-1)}{p}}.
	$$
	Combining this estimate with \eqref{eq-3.4} and applying Lemma \ref{Lem-D}, we obtain
	$$
	\begin{aligned}
		&\int_{-1}^{1}
		|T_{\alpha,\eta}f(r)|^p(1-r^2)^{n-2}\,dr \\
		&\leq2^{p\alpha+n-2}C_{n,\alpha}^{p}
		\left(\frac{\omega_{n-2}}{\omega_{n-1}}\right)^{p-1}\left[B\left(\frac{(n-1)(p-1)}{2p},
		\frac{\alpha+1+\frac{n-1}{p}}{2}\right)\right]^p\int_{\mathbb S^{n-1}}|f(\xi)|^p\,d\sigma(\xi).
	\end{aligned}
	$$
	
	Set $0<\varepsilon<\frac{n-1}{p}$, define
	$$
	f_\varepsilon(\xi)
	=\left(1-\langle\xi,\eta\rangle^2\right)^{
		-\frac{1}{2}\left(\frac{n-1}{p}-\varepsilon\right)}.
	$$
	If $\theta=\arccos\langle\xi,\eta\rangle$,
	then
	$$
	f_\varepsilon(\xi)
	=(\sin\theta)^{-\frac{n-1}{p}+\varepsilon}.
	$$
	By Lemma \ref{Lem-C}, we have that
	\begin{align*}
		\|f_\varepsilon\|_{L^p(\mathbb S^{n-1})}^p
		&=\frac{\omega_{n-2}}{\omega_{n-1}}
		\int_0^\pi(\sin\theta)^{n-2-p\left(\frac{n-1}{p}-\varepsilon\right)}
		\,d\theta \\
		&=\frac{\omega_{n-2}}{\omega_{n-1}}
		\int_0^\pi
		(\sin\theta)^{-1+p\varepsilon}
		\,d\theta \\
		&=\frac{\omega_{n-2}}{\omega_{n-1}}
		B\left(\frac{p\varepsilon}{2},\frac12\right).
	\end{align*}
	Since
	$$
	B\left(\frac{p\varepsilon}{2},\frac12\right)
	=\frac{2}{p\varepsilon}(1+o(1)),
	\qquad \varepsilon\to0^+,
	$$
	we have
	\begin{equation}\label{eq-3.5}
		\lim_{\varepsilon\to0^+}\|f_\varepsilon\|_{L^p(\mathbb S^{n-1})}^p
		=\lim_{\varepsilon\to0^+}
		\frac{2}{p\varepsilon}\frac{\omega_{n-2}}{\omega_{n-1}}(1+o(1)).
	\end{equation}
	
	Set
	$$
	u_\varepsilon(r)
	=(T_{\alpha,\eta}f_\varepsilon)(r).
	$$
	For sufficiently small $\varepsilon>0$ (we can  choose $0<\varepsilon<1+\alpha+\frac{1}{p}$), the same computation yields
	$$
	\lim_{r\to1^-}
	(1-r^2)^{\frac{n-1}{p}-\varepsilon}
	u_\varepsilon(r) \\
	=C_{n,\alpha}\frac{\omega_{n-2}}{\omega_{n-1}}
	2^{\alpha+\frac{n-1}{p}-\varepsilon}
	B\left(\frac{n-1-\left(\frac{n-1}{p}-\varepsilon\right)}{2},\frac{\alpha+1+\frac{n-1}{p}-\varepsilon}{2}\right).
	$$
	In particular,
	$$
	\lim_{\varepsilon\to0^+}
	\lim_{r\to1^-}(1-r^2)^{\frac{n-1}{p}-\varepsilon}
	u_\varepsilon(r) \\
	=C_{n,\alpha}\frac{\omega_{n-2}}{\omega_{n-1}}
	2^{\alpha+\frac{n-1}{p}}
	B\left(
	\frac{(n-1)(p-1)}{2p},
	\frac{\alpha+1+\frac{n-1}{p}}{2}
	\right).
	$$
	Fix $\delta\in(0,1)$. Then there exist $r_\delta\in(0,1)$ and $\varepsilon_\delta>0$ such that, whenever $r_\delta<r<1$ and $0<\varepsilon<\varepsilon_\delta$,
	$$
	\begin{aligned}
		u_\varepsilon(r)
		&\geq(1-\delta)	C_{n,\alpha}
		\frac{\omega_{n-2}}{\omega_{n-1}}
		2^{\alpha+\frac{n-1}{p}-\varepsilon}
		B\left(\frac{n-1-\left(\frac{n-1}{p}-\varepsilon\right)}{2},\frac{\alpha+1+\frac{n-1}{p}-\varepsilon}{2}
		\right) \\
		&\quad\times(1-r^2)^{-\frac{n-1}{p}+\varepsilon}.
	\end{aligned}
	$$
	Therefore
	$$
	\begin{aligned}
		&\int_{r_\delta}^{1}
		|u_\varepsilon(r)|^p(1-r^2)^{n-2}\,dr \\
		&\geq(1-\delta)^p\left[C_{n,\alpha}\frac{\omega_{n-2}}{\omega_{n-1}}2^{\alpha+\frac{n-1}{p}-\varepsilon}
		B\left(\frac{n-1-\left(\frac{n-1}{p}-\varepsilon\right)}{2},\frac{\alpha+1+\frac{n-1}{p}-\varepsilon}{2}
		\right)\right]^p  \\
		&\quad\times\int_{r_\delta}^{1}(1-r^2)^{n-2-p\left(\frac{n-1}{p}-\varepsilon\right)}\,dr.
	\end{aligned}
	$$
	Since $n-2-p\left(\frac{n-1}{p}-\varepsilon\right)
	=-1+p\varepsilon$, by \cite[Lemma  3.2]{AK2026},  
	we have
	$$
	\int_{r_\delta}^{1}
	(1-r^2)^{-1+p\varepsilon}\,dr
	=\frac{1}{2p\varepsilon}(1+o(1)),
	\qquad \varepsilon\to0^+.
	$$
	Since $u_\varepsilon$ is an even  function in $r$, the two endpoints $r=1$ and $r=-1$ give the same asymptotic contribution. Hence
	$$
	\begin{aligned}
		&\lim_{\varepsilon\to0^+}\int_{-1}^{1}
		|u_\varepsilon(r)|^p(1-r^2)^{n-2}\,dr \\
		&\geq\lim_{\varepsilon\to0^+}
		(1-\delta)^p\left[	C_{n,\alpha}
		\frac{\omega_{n-2}}{\omega_{n-1}}
		2^{\alpha+\frac{n-1}{p}-\varepsilon}
		B\left(\frac{n-1-\left(\frac{n-1}{p}-\varepsilon\right)}{2},\frac{\alpha+1+\frac{n-1}{p}-\varepsilon}{2}
		\right)\right]^p\frac{1}{p\varepsilon}(1+o(1)).
	\end{aligned}
	$$
	Dividing  \eqref{eq-3.5} and then letting $\varepsilon\to0^+$ gives
	$$
	\begin{aligned}
		&\liminf_{\varepsilon\to0^+}
		\frac{\int_{-1}^{1}|u_\varepsilon(r)|^p(1-r^2)^{n-2}\,dr
		}{\|f_\varepsilon\|_{L^p(\mathbb S^{n-1})}^p
		} \\
		&\geq(1-\delta)^p2^{p\alpha+n-2}	C_{n,\alpha}^{p}
		\left(\frac{\omega_{n-2}}{\omega_{n-1}}\right)^{p-1}\left[B\left(\frac{(n-1)(p-1)}{2p},
		\frac{\alpha+1+\frac{n-1}{p}}{2}\right)\right]^p.
	\end{aligned}
	$$
	Finally, letting $\delta\to0^+$ proves that the constant in the upper estimate is sharp.

		\vspace*{3mm}
	\subsection*{ Acknowledgement}
	The first author was partly supported by the Scientific Research Fund of Hunan Provincial Education Department under the number 25A0086,
		 NSF of China under the number 12371071 and number 12571081,  and the Key Project of NSF of Hunan Province under the number 2026JJ30002.
	
\end{document}